\documentclass{amsart}
\usepackage{amssymb,amscd,graphicx,epsfig}
\usepackage[colorlinks=true, pdfstartview=FitV, linkcolor=blue, citecolor=blue, urlcolor=blue]{hyperref}

\title[Degree bounds for separating invariants]{Degree bounds for separating invariants}

\author{Martin Kohls and Hanspeter Kraft}
\date{June 22, 2010}

\address{Zentrum Mathematik - M11, Technische Universit\"at M\"unchen,
Bolzmannstrasse 3, D-85748 Garching, Germany}
\email{kohls@ma.tum.de}
\address{Mathematisches Institut,
Universit\"at Basel, Rheinsprung 21, CH-4051 Basel, Switzerland}
\email{Hanspeter.Kraft@unibas.ch}

\newtheorem{thm}{Theorem}
\newtheorem*{thm*}{Theorem}

\newtheorem*{thA}{Theorem A}
\newtheorem*{thB}{Theorem B}
\newtheorem*{thC}{Theorem C}
\newtheorem*{conj*}{Conjecture}
\newtheorem{ques}{Question}

\newtheorem*{prob*}{Problem}
\newtheorem*{satz*}{Satz}

\newtheorem{prop}{Proposition}
\newtheorem{lem}{Lemma}
\newtheorem{cor}{Corollary}
\newtheorem*{cor*}{Corollary}

\theoremstyle{definition}
\newtheorem{defn}{Definition}
\newtheorem{exa}{Example}

\theoremstyle{remark}
\newtheorem*{rem*}{Remark}
\newtheorem{rem}{Remark}

\newcommand{\op}{\operatorname}

\def\Magma{{\sc Magma }}

\newcommand{\name}[1]{\textsc{#1\/}}
\newcommand{\NN}{{\mathbb N}}
\newcommand{\ZZ}{{\mathbb Z}}

\def\reg{\operatorname{reg}}

\newcommand{\simto}{\xrightarrow{\sim}}
\newcommand{\be}{\begin{enumerate}}
\newcommand{\ee}{\end{enumerate}}

\newcommand{\Id}{\op{Id}}

\newcommand{\SLtwo}{{\op{SL}_2(K)}}
\newcommand{\SL}{\op{SL}}

\newcommand{\GL}{\op{GL}}

\newcommand{\Hom}{\op{Hom}}

\newcommand{\Ind}{\op{Ind}}

\newcommand{\Char}{\op{char}}

\def\ord{\operatorname{ord}}

\newcommand{\OOO}{\op{\mathcal O}}
\newcommand{\quot}{/\!\!/}

\newcommand{\bsep}{{\beta_{\text{{\rm sep}}}}}
\DeclareMathOperator{\Spec}{Spec}

\DeclareMathOperator{\Ss}{\mathcal S}

\renewcommand{\phi}{\varphi}
\def \itt #1,#2:{\medskip\item[$\bullet$] %
     page\ \ignorespaces#1, line\ \ignorespaces#2:\ \ignorespaces}

\newcommand{\lab}[1]{\label{#1}}

\begin{document}

\begin{abstract}
If $V$ is a representation of a linear algebraic group $G$, a set $S$ of
$G$-invariant regular functions on $V$ is called {\it separating\/} if the
following holds: {\it If two elements $v,v'\in V$ can be separated by an invariant function, then there is an $f\in S$ such that $f(v)\neq f(v')$.} It is known that there always exist  finite separating sets. Moreover, if the group $G$ is finite, then the invariant functions of degree $\leq |G|$ form a separating set. We show that for a non-finite linear algebraic group $G$ such an upper bound for the degrees of a separating set does not exist.

If $G$ is finite, we define $\bsep(G)$ to be the minimal number $d$ such that for every $G$-module $V$ there is a separating set of degree $\leq d$. We show that for a subgroup $H \subset G$ we have $\bsep(H) \leq \bsep(G)\leq [G:H] \cdot\bsep(H)$, and that $\bsep(G)\leq \bsep(G/H) \cdot \bsep(H)$ in case $H$ is normal. Moreover, we calculate $\bsep(G)$ for some specific finite groups.
\end{abstract}

\maketitle

\vskip1cm
\section{Introduction}
Let $K$ be an algebraically closed field of arbitrary characteristic. Let $G$ be a linear algebraic group and $X$ a $G$-variety, i.e. an affine variety equipped with a (regular) action of $G$, everything defined over $K$. We denote by $\OOO(X)$ the coordinate ring of $X$ and by $\OOO(X)^{G}$ the subring of $G$-invariant regular functions. The following definition is due to \name{Derksen} and \name{Kemper} \cite[Definition 2.3.8]{Derksen2002Computational-i}.

\begin{defn}\lab{Def1}
Let $X$ be a $G$-variety. A subset $S\subset \OOO(X)^{G}$ of the invariant ring of $X$ is called {\it separating\/} (or {\it $G$-separating\/}) if the following holds:
\par
{\it For any pair $x,x' \in X$, if $f(x)\neq f(x')$ for some $f\in\OOO(X)^{G}$ then there is an $h\in S$ such that $h(x)\neq h(x')$.}
\end{defn}

It is known and easy to see that there always exists a finite separating set (see 
\cite[Theorem 2.3.15]{Derksen2002Computational-i}). 

If $V$ is a {\it $G$-module}, i.e. a finite dimensional $K$-vector space with
a regular linear action of $G$, we would like to know a priory bounds for the
degrees of the elements in a separating set. We denote by $\OOO(V)_{d}\subset
\OOO(V)$ the homogeneous functions of degree $d$ (and the zero function), and put $\OOO(V)_{\leq d}:=\bigoplus_{i=0}^{d}\OOO(V)_{i}$.

\begin{defn} For a $G$-module $V$ define
$$
\bsep(G,V) := \min\{d \mid \OOO(V)_{\leq d}\text{ is $G$-separating}\}\in\NN,
$$
and set
$$
\bsep(G):=\sup\{\bsep(G,V) \mid \text{$V$ a $G$-module}\}\in\NN\cup\{\infty\}.
$$
\end{defn}

The main results of this note are the following.
\begin{thA}
The group $G$ is finite if and only if $\bsep(G)$ is finite.
\end{thA}
In order to prove this we will show that $\bsep(K^{+}) = \infty$, that $\bsep(K^{*}) = \infty$, that $\bsep(G)=\infty$ for every semisimple group $G$, and that $\bsep(G^{0})\leq \bsep(G)$ (see section~\ref{alggroups.sec}, Theorem~\ref{alggroup.thm}).

\begin{thB}\lab{finitegroups.thmB}
Let $G$ be a finite group and  $H \subset G$ a subgroup. Then
$$
\bsep(H) \leq \bsep(G) \leq [G:H] \, \bsep(H), \text{ and so }  \bsep(G) \leq |G|.
$$
Moreover, if $H \subset G$ is normal, then
$$
\bsep(G) \leq \bsep(G/H)\, \bsep(H).
$$
\end{thB}
This will be done in section~\ref{reldegreebounds.sec} where we formulate and prove a more precise statement (Theorem~\ref{finitegroups.thm}).

Finally, we have the following explicit results for finite groups.
\begin{thC}
\be
\item Let $\Char K=2$. Then $\bsep(S_{3})=4$.
\item  Let $\Char K=p>0$ and let $G$ be a finite $p$-group. Then $\bsep(G)=|G|$.
\item Let $G$ be a finite cyclic group. Then $\bsep(G)=|G|$.
\item Assume $\Char(K)=p$ is odd, and $r\ge 1$. Then $\bsep(D_{2p^{r}})=2p^{r}$.
\ee
\end{thC}

For a reductive group  $G$ one knows that the condition $f(x)\neq f(x')$ for some invariant $f$ (in Definition~\ref{Def1}) is equivalent to the condition  $\overline{Gx} \cap \overline{Gx'}=\emptyset$, see \cite[Corollary 3.5.2]{Newstead}. This gives rise to the following definition.

\begin{defn}
Let $X$ be a $G$-variety. A $G$-invariant morphism $\phi\colon X \to Y$ where $Y$ is an affine variety is called  {\it separating\/} (or {\it $G$-separating\/}) if the following condition holds: {\it For any pair $x,x'\in X$ such that  $\overline{Gx} \cap \overline{Gx'}=\emptyset$ we have $\phi(x)\neq \phi(x')$.}
\end{defn}

\begin{rem}\lab{rem1}
If $\phi\colon X \to Y$ is $G$-separating and $X' \subset X$ a closed $G$-stable subvariety, then the induced morphism $\phi|_{X'}\colon X' \to Y$ is also $G$-separating.
\end{rem}

\begin{rem} 
Choose a closed embedding $Y\subset K^{m}$ and denote by $\phi_{1},\ldots,\phi_{m} \in \OOO(X)$ the coordinate functions of $\phi\colon X \to Y \subset K^{m}$. {\it If $\phi$ is separating, then $\{\phi_{1},\ldots,\phi_{m}\}$ is a separating set.} The converse holds if $G$ is reductive, but not in general, as shown by the standard representation $K^{+} \to \GL_{2}(K)$ given by $s\mapsto \begin{bmatrix} 1& s\\ &1 \end{bmatrix}$ which does not admit a separating morphism.
\end{rem}


\vskip1cm
\section{Some useful results}\lab{knownresults.sec}
We want to recall some facts about the $\bsep$-values, and compare with the results for the classical $\beta$-values for generating invariants introduced by \name{Schmid} \cite{Schmid1991Finite-groups-a}: 
{\it $\beta(G)$ is the minimal $d\in \NN$ such that,  for every $G$-module $V$, the invariant ring $\OOO(V)^{G}$ is generated by the invariants of degree $\leq d$.}

By \name{Derksen}  and \name{Kemper} 
\cite[Corollary 3.9.14]{Derksen2002Computational-i}, we have $\bsep(G)\le |G|$. This
is in perfect analogy to the Noether bound, which says $\beta(G)\le |G|$ in
the non-modular case, see \cite{Schmid1991Finite-groups-a, FleischmannNoether, FogartyNoether}.
Of course we have $\bsep(G)\le \beta(G)$, so when good upper bounds for
$\beta(G)$ are known, then we have an upper bound for $\bsep(G)$. 

In characteristic zero and in the non-modular case there are
the bounds by \name{Schmid} \cite{Schmid1991Finite-groups-a} and by \name{Domokos}, \name{Heged\"us}, and \name{Sezer}
\cite{DomokosHegedus,SezerNoether} which improve the Noether bound. In particular, \cite{SezerNoether} shows
for non-modular non-cyclic groups $G$ that  $\beta(G)\le \frac{3}{4}|G|$. 

For a linear algebraic group $G$ it is shown by \name{Bryant}, \name{Derksen}  and \name{Kemper}  \cite{Bryant,DerksenKemperGlobal} that 
$\beta(G)<\infty $ if and only if $G$ is finite  and  $p \nmid |G|$ which is
the analogon to our Theorem~A. For further results on degree bounds, we
recommend the overview article of \name{Wehlau} \cite{WehlauNoether}. 

\par\smallskip
The following results will be useful in the sequel.
\begin{prop}\lab{GHseparating.prop}
Let $H \subset G$ be a closed subgroup, $X$ an affine $G$-variety and $Z$ an affine $H$-variety. Let $\iota\colon Z \to X$ be an $H$-equivariant morphism and assume that $\iota^{*}$ induces a surjection $\OOO(X)^{G}\twoheadrightarrow \OOO(Z)^{H}$. If $S \subset \OOO(X)^{G}$ is $G$-separating, then the image $\iota^{*}(S) \subset \OOO(Z)^{H}$ is $H$-separating.
\end{prop}
\begin{proof} Let  $f\in \OOO(Z)^{H}$ and $z_{1},z_{2}\in Z$ such that $f(z_{1})\neq f(z_{2})$. By assumption $f=\iota^{*}(\tilde f)$ for some $\tilde f\in\OOO(X)^{G}$. Put $x_{i}:=\iota(z_{i})$. Then $\tilde f(x_{1})=f(z_{1})\neq f(z_{2}) = \tilde f(x_{2})$. Thus we can find an $h\in S$ such that $h(x_{1})\neq h(x_{2})$. It follows that $\bar h:=\iota^{*}(h) \in \iota^{*}(S)$ and $\bar h(z_{1}) = h(x_{1})\neq h(x_{2})=\bar h(z_{2})$.
\end{proof}

\begin{rem}
In general, the inverse map $(\iota^{*})^{-1}$ does not take $H$-separating sets
to $G$-separating sets. Take $K^{+}\subset \SL_{2}$ as the subgroup of upper triangular unipotent matrices, $X=K^{2}\oplus K^{2}\oplus K^{2}$ the sum of three copies of 
the standard representation of $\SL_{2}$ and $Z=K^{2}\oplus K^{2}$ the sum of two copies of the 
standard representation of $K^{+}$. Then $\iota\colon Z\to X$, $(v,w)\mapsto ((1,0),v,w)$ is $K^{+}$-equivariant and
induces an isomorphism  $\OOO(X)^{\SL_{2}}\simto \OOO(Z)^{K^{+}}$ (Roberts \cite{Roberts}). In fact, 
choosing the coordinates $(x_{0},x_{1},y_{0},y_{1},z_{0},z_{1})$ on $X$ and
$(y_{0},y_{1},z_{0},z_{1})$ on $Y$, we get from the classical description
\cite{ConciniProcesi} of the invariants and covariants of copies of $K^{2}$:
\begin{align*}
\OOO(X)^{\SLtwo} &=K[y_{1}x_{0}-y_{0}x_{1},z_{1}x_{0}-z_{0}x_{1},y_{1}z_{0}-y_{0}z_{1}],\\
\OOO(Y)^{K^{+}}&= K[y_{1},z_{1},y_{1}z_{0}-y_{0}z_{1}],
\end{align*}
and the claim follows, because $\iota^*(x_{0})=1, \iota^*(x_{1})=0$.

Now take $S:=\{y_{1},z_{1},y_{1}(y_{1}z_{0}-y_{0}z_{1}),z_{1}(y_{1}z_{0}-y_{0}z_{1})\}\subset
\OOO(Z)^{K^{+}}$. We show that $S$ is a $K^{+}$-separating set, but $(\iota^{*})^{-1}(S)\subset
\OOO(X)^{\SL_{2}}$ is not $\SL_{2}$-separating. For the first claim one has to
use that if $y_{1}$ and $z_{1}$ both vanish, then the third generator
$y_{1}z_{0}-y_{0}z_{1}$ of the invariant ring $\OOO(Y)^{K^{+}}$ also
vanishes. For the second claim we consider the elements
$v=((0,0),(0,0),(0,0))$ and $v'=((0,0),(1,0),(0,1))$ of $X$, which are separated by the invariants, but not by $(\iota^{*})^{-1}(S)$.
\end{rem}

For the following application recall that for a closed subgroup $H \subset G$ of finite index the {\it induced module\/} $\Ind_{H}^{G}V$ of an $H$-module $V$ is a finite dimensional $G$-module.

\begin{cor}\lab{IndMod.prop}
Let  $H\subset G$ be a closed subgroup of finite index and let $V$ be an $H$-module. Then $\bsep(H,V)\leq \bsep(G,\Ind_{H}^{G}V)$. In particular,  $\bsep(H)\leq \bsep(G)$. 
\end{cor}

\begin{proof} By definition, $\Ind_{H}^{G}V$ contains $V$ as an $H$-submodule in a canonical way.  
If $n:=[G:H]$ and $G = \bigcup_{i=1}^{n}g_{i}H$, then $\Ind_{H}^{G}V =
\bigoplus_{i=1}^{n}g_{i}V$. Moreover, the inclusion $\iota:
V\hookrightarrow\Ind_{H}^{G}V$ induces a surjection
$\iota^{*}: \OOO(\Ind_{H}^{G}(V))^{G}\twoheadrightarrow \OOO(V)^{H}, \,\,
f\mapsto f|_{V}$. In fact, for  $f\in\OOO(V)^{H}_{+}$, a preimage
$\tilde{f}$ is given by
$\tilde{f}(g_{1}v_{1},\ldots,g_{n}v_{n}):=\sum_{i=1}^{n}f(v_{i})$, $v_{i}\in V$, which is
easily seen to be  $G$-invariant. Now the claim follows from
Proposition~\ref{GHseparating.prop} above, because the restriction map $\iota^{*}$ is linear and so preserves degrees.
\end{proof}

\begin{prop}[{\name{Derksen}  and \name{Kemper}  \cite[Theorem 2.3.16]{Derksen2002Computational-i}}]
\lab{separatingSubmodule.prop}
Let $G$ be a reductive group, $V$ a $G$-module und $U\subset V$ a
submodule. The restriction map $\OOO(V)\rightarrow \OOO(U)$, $f\mapsto f|_{U}$ takes every separating
set of $\OOO(V)^{G}$ to a separating set of $\OOO(U)^{G}$. In particular, we have
\[
\bsep(G,U)\le\bsep(G,V).
\] 
\end{prop}
Let us mention here that in positive characteristic the restriction map is in general not surjective when restriced to the invariants, and so a generating set is not necessarily mapped onto  a generating set.

We finally remark that for finite groups there is a $G$-module $V$ such $\bsep(G,V)=\bsep(G)$. The same holds for the $\beta$-values in characteristic zero.

\begin{prop}\lab{betasepExplicit}
Let $G$ be a finite group group and $V_{\reg}=KG$ its regular
representation. Then
\[
\bsep(G)=\bsep(G,V_{\reg}).
\]
\end{prop}

In fact, every $G$-module $V$ can be embedded as a submodule in $V_{\reg}^{\dim V}$. 
Since, by  \cite[Corollary~3.7]{DraismaSeparating},  $\bsep(G,V^{m}) = \bsep(G,V)$ for any $G$-module $V$  and every positive integer $m$, the claim follows from Proposition~\ref{separatingSubmodule.prop}.

\vskip1cm
\section{The case of non-finite algebraic groups}\lab{alggroups.sec}

In this section we prove the following theorem which is equivalent to Theorem~A from the first section.
\begin{thm}\lab{alggroup.thm}
For any non-finite linear algebraic group $G$ we have $\bsep(G) = \infty$.
\end{thm}

We start with the additive group $K^{+}$. Denote by $V=Ke_{0}\oplus
Ke_{1}\simeq K^{2}$ the standard 2-dimensional $K^{+}$-module: $s\cdot
e_{0}:=e_{0},\,\,\, s\cdot e_{1}:=se_{0}+e_{1}$ for $s\in K^{+}$. If $\Char
K=p>0$ we can ``twist'' the module $V$ with the Frobenius map $F^{n}\colon
K^{+}\to K^{+}, s \mapsto s^{p^{n}}$ to obtain another $K^{+}$-module which we
denote by $V_{F^{n}}$.

\begin{prop}\lab{prop1}
Let $\Char K = p>0$ and consider the $K^{+}$-module $W:=V \oplus
V_{F^{n}}$. We write $\OOO(W)=K[x_{0},x_{1},y_{0},y_{1}]$. Then $\OOO(W)^{K^{+}}= 
K[x_{1},y_{1},x_{0}^{p^{n}}y_{1}- x_{1}^{p^{n}}y_{0}]$. In particular, $\bsep(K^{+},W) = p^{n}+1$ and so $\bsep(K^{+}) = \infty$.
\end{prop}

\begin{proof} It is easy to see that $f:=x_{0}^{p^{n}}y_{1}- x_{1}^{p^{n}}y_{0}$ is $K^{+}$-invariant. Define the $K^{+}$-invariant morphism
$$
\pi \colon W \to K^{3}, \ \  w=(a_{0},a_{1},b_{0},b_{1})\mapsto (a_{1},b_{1},a_{0}^{p^{n}}b_{1}- a_{1}^{p^{n}}b_{0}).
$$
Over the affine open set $U:= \{(c_{1},c_{2},c_{3})\in K^{3} \mid c_{1}\neq
0\}$, the induced map $\pi^{-1}(U) \to U$ is a trivial $K^{+}$-bundle. In
fact, the morphism $\rho\colon U \to \pi^{-1}(U)$ given by
$(c_{1},c_{2},c_{3}) \mapsto (0,c_{1},-c_{1}^{-p^{n}}c_{3},c_{2})$ is a
section of $\pi$, inducing a $K^{+}$-equivariant isomorphism $K^{+}\times U
\simto \pi^{-1}(U),\,\, (s,u)\mapsto s\cdot\rho(u)$. This implies that $\OOO(W)^{K^{+}}_{x_{1}}=K[x_{1},x_{1}^{-1},y_{1},f]$, hence $\OOO(W)^{K^{+}}=K[x_{0},x_{1},y_{0},y_{1}] \cap K[x_{1},x_{1}^{-1},y_{1},f]$, and the claim follows easily.
\end{proof}

If $K$ has characteristic zero, we need a different argument. Denote by $V_{n}:=S^{n}V$ the $n$th symmetric power of the standard $K^{+}$-module $V=Ke_{0}\oplus Ke_{1}$ (see above). This module is cyclic of dimension $n+1$, i.e. $V_{n}= \langle K^{+}v_{n}\rangle$ where $v_{n}:=e_{1}^{n}$, and for any $s\in K^{+}, s\neq 0$, the endomorphism $v\mapsto sv-v$ of $V_{n}$ is nilpotent of rank $n$. In particular, $V_{n}^{K^{+}} = Kv_{0}$ where $v_{0}:=e_{0}^{n}\in V_{n}$.

\begin{rem} \lab{cyclicmodules.rem}
For $q\geq 1$ consider the {\it $q$th symmetric power $S^{q}V_{n}$} of the module $V_{n}$. Then the cyclic submodule $\langle K^{+}v_{n}^{q}\rangle\subset S^{q}V_{n}$ generated by $v_{n}^{q}$ is $K^{+}$-isomorphic to $V_{qn}$, and 
$\langle K^{+}v_{n}^{q}\rangle^{K^{+}} = Kv_{0}^{q}$. One way to see this is by remarking that the modules $V_{n}$ are $\SLtwo$-modules in a natural way, and then to use representation theory of $\SLtwo$.
\end{rem}

\begin{prop}\lab{prop2}
Let $\Char K = 0$. Consider the $K^{+}$-module $W = V^{*}\oplus V_{n}$ and the two vectors
$w:=(x_{0},v_{0})$ and $w':=(x_{0},0)$ of $W$. Then there is a
$K^{+}$-invariant function $f\in \OOO(W)^{K^{+}}$ separating $w$ and $w'$, and any such $f$ has degree $\deg f\geq n+1$. In particular, $\bsep(K^{+},W) \geq n+1$, and so $\bsep(K^{+}) = \infty$.
\end{prop}

\begin{proof} Let $U_{1},U_{2}$ be two finite dimensional vector spaces. There is a canonical isomorphism
$$
\Psi\colon \OOO(U_{1}^{*}\oplus U_{2})_{(p,q)}  \simto \Hom(S^{q}U_{2}, S^{p}U_{1})
$$
where $ \OOO(U_{1}^{*}\oplus U_{2})_{(p,q)}$ denotes the subspace of those regular functions on $U_{1}^{*}\oplus U_{2}$ which are bihomogeneous of degree $(p,q)$. If $F = \Psi(f)$, then for any $x\in U_{1}^{*}$ and $u\in U_{2}$ we have
$$
f(x,u) = x^{p}(F(u^{q})).
$$
(Since we are in characteristic 0 we can identify $S^{p}(U_{1}^{*})$ with $(S^{p}U_{1})^{*}$.)
Moreover, if $U_{1},U_{2}$ are $G$-modules, then $\Psi$ is $G$-equivariant and induces an isomorphism between the $G$-invariant bihomogeneous functions and the $G$-linear homomorphisms:
$$
\Psi\colon \OOO(U_{1}^{*}\oplus U_{2})_{(p,q)}^{G}  \simto \Hom_{G}(S^{q}U_{2}, S^{p}U_{1}).
$$
For the $K^{+}$-module $W = V^{*}\oplus V_{n}$ we thus obtain an isomorphism
$$
\Psi\colon \OOO(V^{*}\oplus V_{n})_{(p,q)}^{K^{+}}  \simto \Hom_{K^{+}}(S^{q}V_{n}, S^{p}V).
$$
Putting $p=n$ and $q=1$ and defining $f\in  \OOO(V^{*}\oplus V_{n})_{(n,1)}^{K^{+}}$ by $\Psi(f) = \Id_{V_{n}}$, we get $f(w) = f(x_{0},v_{0}) = x_{0}^{n}(v_{0})=x_{0}^{n}(e_{0}^{n})\neq 0$, and $f(w')=f(x_{0},0)=0$. Hence $w$ and $w'$ can be separated by invariants.

Now let $f$ be a $K^{+}$-invariant separating $w$ and $w'$ where $\deg f = d$. We can clearly assume that $f$ is bihomogeneous, say  of degree
$(p,q)$ where $p+q=d$. Because $f$ must depend on  $V_{n}$, we have $q \geq 1$.
Hence $f(w') = f(x_{0},0)=0$, and so $f(w)=f(x_{0},v_{0})\neq 0$. This implies
for $F:=\Psi(f)$ that $F(v_{0}^{q})\neq 0$. Now it follows from Remark~\ref{cyclicmodules.rem} above  that $F$ induces an injective map of $\langle K^{+}v_{n}^{q} \rangle$ into $S^{p}V$, and so
$$
p+1 = \dim S^{p}V \geq \dim \langle K^{+}v_{n}^{q} \rangle = qn+1 \geq n+1.
$$
Hence $\deg f = p+q \geq n+1$.
\end{proof}

To handle the general case we use the following construction. Let $G$ be an algebraic group and $H \subset G$ a closed  subgroup. We assume that $H$ is reductive. For an affine $H$-variety $X$ we define
$$
G \times^{H} X := (G\times X)\quot H := \Spec(\OOO(G \times X)^{H})
$$
where $H$ acts (freely) on the product $G\times X$ by $h(g,x):= (gh^{-1},hx)$, commuting with the action of $G$ by left multiplication on the first  factor. We denote by $[g,x]$ the image of $(g,x) \in G \times X$
in the quotient $G \times^{H} X$. 

The following is well-known. It follows from general results from geometric invariant theory, see e.g.  \cite{Mumford1994Geometric-invar}.
\be
\item The canonical morphism $G \times^{H} X \to G/H$, $[g,x]\mapsto gH$, is a fiber bundle (in the \'etale topology) with fiber $X$.
\item If the action of $H$ on $X$ extends to an action of $G$, then $G \times^{H} X \simto G/H \times X$ where $G$ acts diagonally on $G/H \times X$ (i.e. the fiber bundle is trivial).
\item The canonical morphism $\iota\colon X \hookrightarrow G \times^{H} X$ given by $x\mapsto [e,x]$ is an $H$-equivariant closed embedding.
\ee

\begin{lem}\lab{lem1}
If $\phi\colon G \times^{H} X \to Y$ is $G$-separating, then the composite morphism  $\phi\circ\iota \colon X \to Y$ is $H$-separating. Moreover, if $S\subset \OOO(G \times^{H} X)^{G}$ is a $G$-separating set, then its image $\iota^{*}(S) \subset \OOO(X)^{H}$ is $H$-separating.
\end{lem}

\begin{proof} For $x\in X$ we have $\overline{G[e,x]} = [G,\overline{Hx}]$. Therefore, if $\overline{Hx} \cap \overline{Hx'}=\emptyset$, then $\overline{G[e,x]} \cap \overline{G[e,x']} = \emptyset$ and so $\phi\circ\iota(x) = \phi([e,x]) \neq \phi([e,x']) = \phi\circ\iota(x')$. The second claim follows from Proposition~\ref{GHseparating.prop}, because $\OOO(G \times^{H} X)^{G} = \OOO(G \times X)^{G\times H} = \OOO(X)^{H}$ and so $\iota^{*}$ induces an isomorphism $\OOO(G\times^{H}X)^{G} \simto \OOO(X)^{H}$.
\end{proof}

Now let $V$ be a $G$-module and  $X:=V|_{H}$, the underlying $H$-module. Choose a closed $G$-equivariant embedding $G/H \simto Gw_{0} \hookrightarrow W$ where $W$ is a $G$-module. Then we get the following composition of closed embeddings where the first one is $H$-equivariant and the remaining are $G$-equivariant:
$$
\mu\colon V|_{H} \hookrightarrow G \times^{H} V \simto G/H \times V \hookrightarrow W \times V.
$$
The map $\mu$ is given by $\mu(v) = (w_{0},v)$. It follows from Lemma~\ref{lem1} and Remark~\ref{rem1}  that for any $G$-separating morphism $\phi\colon W\times V \to Y$ the composition $\phi\circ\mu\colon V|_{H} \to Y$ is $H$-separating. In particular, if $G$ is reductive, then for any $G$-separating set $S\subset \OOO(W\times V)$ the image $\mu^{*}(S) \subset \OOO(V)^{H}$ is $H$-separating. Since $\deg \mu^{*}(f) \leq \deg f$ this implies the following result.

\begin{prop}\lab{subgroup.prop}
Let $G$ be a reductive group, $H\subset G$ a closed reductive subgroup and $V'$ an $H$-module. If $V'$ is isomorphic to an $H$-submodule of a $G$-module $V$, then
$$
\bsep(H,V') \leq \bsep(G).
$$
\end{prop}

Now we can prove the main result of this section,

\begin{proof}[Proof of Theorem~\ref{alggroup.thm}]
By Proposition~\ref{IndMod.prop} we can assume that $G$ is connected.

\smallskip
(a) Let $G$ be semisimple, $T \subset G$ a maximal torus and $B\supset T$ a Borel subgroup. If $\lambda\in X(T)$ is dominant we denote by $E^{\lambda}$ the Weyl-module of $G$ of highest weight $\lambda$, and by $D^{\lambda}\subset E^{\lambda}$  the highest weight line. Choose a one-parameter subgroup $\rho\colon K^{*}\to T$ and define $k_{0}\in\ZZ$ by $\rho(t)u = t^{k_{0}}\cdot u$ for $u\in D^{\lambda}$. For any $n\in\NN$ put
$$
V_{n}' := (D^{\lambda})^{*}\oplus D^{n\lambda} \subset V_{n}:= (E^{\lambda})^{*}\oplus E^{n\lambda}.
$$
Then $V_{n}'$ is a two-dimensional $K^{*}$-module with weights $(-k_{0},nk_{0})$. Hence  $\OOO(V_{n}')^{K^{*}}$ is generated by a homogeneous invariant of degree $n+1$ and so $\bsep(K^{*},V_{n}')=n+1$. Now Proposition ~\ref{subgroup.prop} implies
$$
n+1 = \bsep(K^{*},V_{n}') \leq \bsep(G)
$$
and the claim follows. In addition, we have also shown that $\bsep(K^{*}) = \infty$.

\smallskip
(b) If $G$ admits a non-trivial character $\chi\colon G \to K^{*}$ then the claim follows because $\bsep(G) \geq \bsep(K^{*})=\infty$, as we have seen in (a).

\smallskip
(c) If the character group of $G$ is trivial, then either $G$ is unipotent or
there is a surjective homomorphism $G\to H$ where $H$ is semisimple (use
\cite[Corollary 8.1.6 (ii)]{SpringerNew}). In the first case there is a surjective homomorphism $G \to K^{+}$ and the claim follows from 
Proposition~\ref{prop1} and Proposition~\ref{prop2}. In the second case the claim follows from (a). 
\end{proof}

\vskip1cm
\section{Relative degree bounds}\lab{reldegreebounds.sec}

In this section all groups are finite. We want to prove the following result which covers Theorem~B from the first section.

\begin{thm}\lab{finitegroups.thm}
Let $G$ be a finite group, $H \subset G$ a subgroup, $V$ a $G$-module and $W$ an $H$-module. Then
$$
\bsep(H,W) \leq \bsep(G,\Ind_{H}^{G}W)\quad\text{and}\quad \bsep(G,V)  \leq [G:H] \, \bsep(H,V).
$$
In particular
$$
\bsep(H) \leq \bsep(G) \leq [G:H] \, \bsep(H), \text{ and so }  \bsep(G) \leq |G|.
$$
Moreover, if $H \subset G$ is normal, then
$$
\bsep(G) \leq \bsep(G/H)\, \bsep(H).
$$
\end{thm}
Note that the inequalities $\bsep(G,V)\le [G:H]\bsep(H,V)$ and $\bsep(G)\le |G|$ were
already proved by \name{Derksen}  and \name{Kemper}  (\cite[Corollary 24]{KemperSeparating}, \cite[Corollary 3.9.14]{Derksen2002Computational-i}).

The proof needs some preparation. Let $V,W$ be finite dimensional vector spaces and $\phi\colon V \to W$ a morphism.

\begin{defn}\lab{degree.def} The {\it degree of $\phi$} is defined in the
  following way, generalizing the degree of a polynomial function. Choose a
  basis $(w_{1},\ldots,w_{m})$ of $W$, so that $\phi(v) = \sum_{j=1}^{m}f_{j}(v)
  w_{j}$ for $v\in V$. Then
$$
\deg\phi := \max\{\deg f_{j}|\quad j=1,\ldots,m \}.
$$
It is easy to see that this is independent of the choice of a basis.
\end{defn}

If $V$ is a $G$-module and $\phi\colon V \to W$ a separating morphism, then $\bsep(G,V) \leq \deg\phi$. Moreover,  there is a separating morphism $\phi\colon V \to W$ for some $W$ such that $\bsep(G,V) = \deg\phi$. 

For any (finite dimensional) vector space $W$ we regard $W^{d}=W\otimes K^{d}$ as the direct sum of $\dim W$ copies of the standard $\Ss_{d}$-module $K^{d}$. In this case we have the following result due to \name{Draisma}  , \name{Kemper}  and \name{Wehlau} \cite[Theorem 3.4]{DraismaSeparating}.

\begin{lem}\lab{sepsym.lem}
The polarizations of the elementary symmetric functions form an $\Ss_{d}$-separating set of $W^{d}$. In particular, there is an $\Ss_{d}$-separating morphism $\psi_{W}\colon W^{d}\to K^{N}$ of degree $\leq d$.
\end{lem}
Recall that the polarizations of a function $f\in\OOO(U)$ to $n$ copies of $U$ are defined in the following way. Write
$$
f(t_{1}u_{1}+t_{2}u_{2}+\cdots+t_{n}u_{n}) = \sum_{i_{1},i_{2},\ldots,i_{n}} t_{1}^{i_{1}}t_{2}^{i_{2}}\cdots t_{n}^{i_{n}}
 f_{i_{1}i_{2}\ldots i_{n}}(u_{1},u_{2},\ldots,u_{n})
$$
Then the functions $f_{i_{1}i_{2}\ldots i_{n}}(u_{1},u_{2},\ldots,u_{n})\in\OOO(U^{n})$ are called {\it polarizations of} $f$. Clearly, $\deg f_{i_{1}i_{2}\ldots i_{n}} \leq \deg f$. Moreover, if $U$ is a $G$-module and $f$ a $G$-invariant, then all 
$f_{i_{1}i_{2}\ldots i_{n}}$ are $G$-invariants with respect to the diagonal action of $G$ on $U^{n}$.

\begin{proof}[Proof of Theorem~\ref{finitegroups.thm}]
The first inequality $\bsep(H,W) \leq \bsep(G,\Ind_{H}^{G}W)$ is shown in Proposition~\ref{IndMod.prop}.

Let $V$ be a $G$-module, $v,w\in V$, and let $\phi\colon V \to W$ be an $H$-separating morphism of degree $\bsep(H,V)$. Consider the partition of $G$ into $H$-right cosets: $G = \bigcup_{i=1}^{d}H g_{i}$ where $d:=[G:H]$. Define the following morphism
$$
\begin{CD}
\bar\phi\colon V @>{\tilde\phi}>> W^{d} @>{\psi_{W}}>> K^{N}
\end{CD}
$$
where $\tilde\phi(v):=(\phi(g_{1}v),\ldots,\phi(g_{d}v))$ and $\psi_{W}\colon W^{d}\to K^{N}$ is the separating morphism from Lemma~\ref{sepsym.lem}.

We claim that $\bar\phi$ is $G$-separating. In fact, for $g\in G$ define the permutation $\sigma\in \Ss_{d}$ by $Hg_{i}g= Hg_{\sigma(i)}$, i.e. $g_{i}g = h_{i}g_{\sigma(i)}$ for a suitable $h_{i}\in H$. Then $\phi(g_{i}gv) = \phi(h_{i}g_{\sigma(i)}v) = \phi(g_{\sigma(i)}v)$ and so $\tilde\phi(gv) = \sigma^{-1}\tilde\phi(v)$. This shows that $\bar\phi$ is $G$-invariant. 

Assume now that $gv \neq w$ for all $g \in G$. This implies that $hg_{i}v\neq w$ for all $h\in H$ and $i=1,\ldots d$, and so $\phi(g_{i}v)\neq\phi(w)$ for $i=1,\ldots,d$, because $\phi$ is $H$-separating. As a consequence, $\tilde\phi(v) \neq \sigma\tilde\phi(w)$ for all permutations $\sigma\in \Ss_{d}$, hence $\bar\phi(v)\neq\bar\phi(w)$, because $\psi_{W}$ is $\Ss_{d}$-separating, and so $\bar\phi$ is $G$-separating.

For the degree  we get $\deg\bar\phi\leq \deg \psi_{W} \cdot \deg\tilde\phi \leq d \cdot \deg\phi=[G:H] \bsep(H,V)$.
This shows that 
$$
\bsep(G,V) \leq [G:H] \bsep(H,V).
$$
If $H\subset G$ is normal we can find an $H$-separating morphism $\phi\colon V \to W$ of degree $\bsep(H,V)$ such that $W$ is a $G/H$-module and $\phi$ is $G$-equivariant. Now choose an $G/H$-separating morphism $\psi\colon W \to U$ of degree  $\bsep(G/H,W)$. Then the composition $\psi\circ\phi\colon V \to U$ is $G$-separating of degree $\leq \deg\psi \cdot \deg\phi$. Thus
$$
\bsep(G,V) \leq \bsep(G/H,W)\,\bsep(H,V) \leq \bsep(G/H)\, \bsep(H),
$$
and the claim follows.
\end{proof}

\vskip1cm
\section{Degree bounds for some finite groups}

In principal, Proposition~\ref{betasepExplicit} allows to compute $\bsep(G)$ for any finite
group $G$. Unfortunately, the invariant ring $\OOO(V_{\reg})^{G}$ does not behave
well in a computational sense. We have been able to compute $\bsep(G)$ with \Magma
\cite{Magma} and the algorithm of \cite{KemperCompRed} in just one case (computation time about 20 minutes):

\begin{prop}[\Magma and Proposition~\ref{betasepExplicit}]\lab{MagmaBoundS3}
Let $\Char K=2$. Then $\bsep(S_{3})=4$.
\end{prop}

\begin{prop}\lab{betapGroup}
Let $\Char K=p>0$ and let $G$ be a $p$-group. Then $\bsep(G)=|G|$.
\end{prop}
\begin{proof}
Let us start with a general remark. Let $G$ be an arbitrary finite group, and let $V$ be a {\it permutation module of $G$}, i.e. there is a basis $(v_{1},v_{2}, \ldots, v_{n})$ of $V$ which is permuted under $G$. Then the invariants are linearly spanned by the {\it orbit sums} $s_{m}$ of the monomials $m=x_{1}^{i_{1}}x_{2}^{i_{2}}\cdots x_{n}^{i_{n}}\in \OOO(V) = K[x_{1},x_{2},\ldots,x_{n}]$ which are defined in the usual way:
\[
s_{m}:=\sum_{f\in Gm}f
\]
The value of $s_{m}$ on the fixed point $v:=v_{1}+ v_{2}+\cdots + v_{n}\in V$ equals $|Gm|$. Hence, $s_{m}(v)=0$ if $p$ divides the index $[G:G_{m}]$ of the stabilizer $G_{m}$ of $m$ in $G$. It follows that for a $p$-group $G$ we have $s_{m}(v)\neq 0$ if and only if $m$ is  invariant under $G$.

If, in addition, $G$ acts transitively on the basis $(v_{1},v_{2}, \ldots, v_{n})$, then an invariant monomial $m$ is a power of $x_{1}x_{2}\cdots x_{n}$, and  thus has degree $\ell n \geq \dim V$. If we apply this to the regular representation, the claim follows.
\end{proof}

With Corollary~\ref{IndMod.prop} we get the next result.
\begin{cor}
Let $\Char K=p>0$ and $G$ be a group of order $rp^{k}$ with $(r,p)=1$. Then $\bsep(G)\ge p^{k}$.
\end{cor}

\begin{prop}\lab{betaCyclicGroup}
Let $G$ be a cyclic group. Then $\bsep(G)=|G|$.
\end{prop}

\begin{proof}
Let $|G|=r p^{k}$ where $(r,p)=1$  and choose two elements $g,h \in G$ of order $r$ and $q:=p^{k}$, respectively, so that  $G=\langle g,h\rangle$. We define a linear action of $G$ on  $V := \bigoplus_{i=1}^{q} K v_{i}$ by 
$$
gv_{i}:= \zeta\cdot  v_{i} \text{ and } hv_{i}:= v_{i+1} \text{ for } i=1,\ldots q
$$
where $\zeta\in K$ is a primitive $r$th root of unity and $v_{q+1}:=v_{1}$. We claim that the $G$-invariants $\OOO(V)^{G}$ are linearly spanned by the orbit sums $s_{m}$ where $r | \deg m$. In fact, $\OOO(V)^{\langle g\rangle}$ is linearly spanned by the monomials of degree $\ell r$ ($\ell\geq 0$), and the subgroup $H:=\langle h\rangle\subset G$ permutes these monomials. 

Now look again at the element $v: =v_{1}+ v_{2}+\cdots + v_{q}\in V$. If $r | \deg m$ then $s_{m}(v) = |Hm|$, and this is non-zero if and only if the monomial $m$ is invariant under  $H$. This implies that  $m$ is a power of $x_{1}x_{2}\cdots x_{q}$. Since the degree of $m$ is also a multiple of $r$ we finally get $\deg s_{m}\geq r q = |G|$.
\end{proof} 

\begin{cor}\lab{OrderDegBound}
Let $G$ be a finite group. Then we have
\[
\bsep(G)\ge \max_{g\in G}(\ord g).
\]
\end{cor}

Let $D_{2n}=\langle \sigma,\rho\rangle$ denote the dihedral group of order
$2n$ with $\ord(\sigma)=2$, $\ord(\rho)=n$ and
$\sigma\rho\sigma^{-1}=\rho^{-1}$. 

\begin{prop}\lab{Dihedral2pr}
Assume that $\Char(K)=p$ is an odd prime, and let $r\ge 1$. Then $\bsep(D_{2p^{r}})=2p^{r}$.
\end{prop}

Note that if $\Char(K)=p=2$, then $D_{2p^{r}}$ is a $2$-group, so
$\bsep(D_{2^{r+1}})=2^{r+1}$ by Proposition~\ref{betapGroup}. We conjecture that for $\Char(K)=2$ and $p$ an odd prime, we have $\bsep(D_{2p})=p+1$, which would fit with Proposition~\ref{MagmaBoundS3}.

\begin{proof}
Put $q=p^{r}$ and define a linear action of $D_{2p^{r}}$ on $V:=\bigoplus_{i=0}^{q-1} Kv_{i}$ by
$$
\rho v_{i}= v_{i+1} \text{ and } \sigma v_{i}= -v_{-i} \text{ for } i=0,1,\ldots,q-1
$$
where $v_{j}=v_{i}$ if $j\equiv i \!\mod q$ for $i,j\in\ZZ$. As before, the invariants under
$H:=\langle \rho \rangle$ are linearly spanned by the orbit sums
$s_{m}:=\sum_{f\in Hm}f$ of the monomials $m=x_{0}^{i_{0}}x_{1}^{i_{1}}\cdots
x_{q-1}^{i_{q-1}} \in \OOO(V) = K[x_{0},x_{1},\ldots,x_{q-1}]$. 
Thus, the $D_{2p^{r}}$-invariants are linearly spanned by the functions
$\{s_{m}+ \sigma s_{m}\mid m \text{ a monomial}\}$.

For  $v:=v_{0}+ v_{1}+ \cdots + v_{q-1}$ we get $\sigma
s_{m}(v)=s_{m}(\sigma v)=(-1)^{\deg m}s_{m}(v)$. Therefore, $s_{m}+ \sigma s_{m}$
is non-zero on $v$ if and only if $s_{m}(v)\neq 0$ and the degree of $m$ is
even. As in the proof of Proposition~\ref{betaCyclicGroup},  $s_{m}(v)\ne 0$ implies
that $m$ is a power of $x_{0}x_{1}\cdots x_{q-1}$ which has to be an even power since
$q$ is odd. Thus, for $m:=(x_{0}x_{1}\cdots x_{q-1})^{2}$, $s_{m}+ \sigma s_{m}= 2m$ is an invariant of smallest possible degree, namely $2q$, which does not vanish on $v$.
\end{proof}

Let $I_{H}:=\OOO(V)^{G}_{+}\OOO(V)$ denote the \emph{Hilbert-ideal}, i.e. the ideal
in $\OOO(V)$ generated by all homogeneous invariants of positive degree. It is
conjectured by \name{Derksen}  and \name{Kemper}  that 
$I_{H}$ is generated by invariants of positive degree $\leq |G|$, see \cite[Conjecture 3.8.6 (b)]{Derksen2002Computational-i}. The following corollary shows that this conjectured bound can not be sharpened in general.

\begin{cor}
Let $\Char K=p$ and $G$ a $p$-group (with $p>0$), or a cyclic group, or $G=D_{2p^{r}}$ with
$p$ odd. Then there exists a $G$-module $V$ such that $I_{H}$ is not
generated by homogeneous invariants of positive degree strictly less than $|G|$. 
\end{cor} 

\begin{proof} In the proofs of the Propositions~\ref{betapGroup},
\ref{betaCyclicGroup} and \ref{Dihedral2pr} respectively, we constructed a
$G$-module $V$ and a non-zero $v\in V$ such that $f(v)=0$ for all homogeneous $f\in
\OOO(V)^{G}$ of positive degree strictly less than $|G|$, but such that there
exists a homogeneous $f\in \OOO(V)^{G}$ of degree $|G|$ with $f(v)\ne 0$. This shows
that $f \notin \OOO(V)^{G}_{+,<|G|}\OOO(V)$. \end{proof}

Now we use relative degree bounds for separating
invariants and good degree bounds for generating invariants of non-modular
groups, that appear as a subquotient, to get improved degree bounds for
separating invariants in the modular case.

\begin{prop}\lab{SezerAnalog}
Let $\Char K=p$ and $G$ be a finite group. Assume there exists a chain of
subgroups $N\subset H\subset G$ such that $N$ is a normal subgroup of $H$ and such that $H/N$
is non-cyclic of order $s$ coprime to $p$. Then \[\bsep(G)\le \left\{\begin{array}{cc}
\frac{3}{4}|G| & \text{ in case } s \text{ is even}\\
\frac{5}{8}|G| & \text{ in case } s \text{ is odd.}
\end{array}\right.\]
\end{prop}

\begin{proof} By \name{Sezer} \cite{SezerNoether}, for a non-cyclic non-modular group $U$, we
have $\beta(U)\le \frac{3}{4}|U|$ in case $|U|$ is even, and $\beta(U)\le
\frac{5}{8}|U|$ in case $|U|$ is odd. We now assume $s$ is even; the other
case is essentially the same. Since $\bsep(U)\le\beta(U)$ always
holds, we get by using Theorem~\ref{finitegroups.thm}
\begin{eqnarray*}
\bsep(G) \leq \bsep(H)[G:H]\leq \bsep(N)\bsep(H/N)[G:H]\\
\leq \beta(H/N)[G:H]|N|\leq
\frac{3}{4}[H:N][G:H]|N|=\frac{3}{4}|G|.
\end{eqnarray*}
\end{proof}

\begin{exa}
Assume $p=3$ and $G=A_{4}$. The Klein four group is a non-cyclic non-modular
subgroup of even order. We get $\bsep(A_{4})\le
\frac{3}{4}|A_{4}|=9$. Application of Theorem~\ref{finitegroups.thm}
shows $\bsep(A_{4}\times A_{4})\le \bsep(A_{4})^{2}\le 81$.
\end{exa}

\begin{exa}
Let $D_{2n}$ be the dihedral group of order $2n$. We know
$n\le\bsep(D_{2n})$ by Corollary \ref{OrderDegBound}. Assume $\Char K=p\ne 2$ and
$n=p^{r}m$ with $p,m$ coprime and $m>1$. Then $D_{2n}$ has the non-cyclic
subgroup $D_{2m}$ of even order, so $\bsep(D_{2n})\le
\frac{3}{4}2n=\frac{3}{2}n$. So the only dihedral groups, to which the proposition
above does not apply, are those of the form $D_{2p^{r}}$, which are covered
by Proposition~\ref{Dihedral2pr}. 
\end{exa}

We end this section with two questions:

\begin{ques}
Which finite groups $G$ satisfy $\bsep(G)=|G|$?
\end{ques}

\begin{ques}
Which finite groups $G$ do not have a non-cyclic non-modular subquotient?
\end{ques}

The dihedral groups of  Proposition~\ref{Dihedral2pr} satisfy this
property, and we get $\bsep(G)=|G|$ for those groups. But in characteristic $2$,
$\bsep(S_{3})<|S_{3}|$ by Proposition~\ref{MagmaBoundS3}, so the answer to the
second question only partially helps to solve the first one.

\par\bigskip\bigskip
\bibliography{KohlsKraft}

\def\polhk#1{\setbox0=\hbox{#1}{\ooalign{\hidewidth
  \lower1.5ex\hbox{`}\hidewidth\crcr\unhbox0}}} \def\cprime{$'$}
  \def\polhk#1{\setbox0=\hbox{#1}{\ooalign{\hidewidth
  \lower1.5ex\hbox{`}\hidewidth\crcr\unhbox0}}} \def\cprime{$'$}
  \def\polhk#1{\setbox0=\hbox{#1}{\ooalign{\hidewidth
  \lower1.5ex\hbox{`}\hidewidth\crcr\unhbox0}}} \def\cprime{$'$}
\providecommand{\bysame}{\leavevmode\hbox to3em{\hrulefill}\thinspace}
\providecommand{\MR}{\relax\ifhmode\unskip\space\fi MR }
\providecommand{\MRhref}[2]{%
  \href{http://www.ams.org/mathscinet-getitem?mr=#1}{#2}
}
\providecommand{\href}[2]{#2}
\begin{thebibliography}{DKW08}

\bibitem[BCP97]{Magma}
Wieb Bosma, John Cannon, and Catherine Playoust, \emph{The {M}agma algebra
  system. {I}. {T}he user language}, J. Symbolic Comput. \textbf{24} (1997),
  no.~3-4, 235--265, Computational algebra and number theory (London, 1993).
  \MR{MR1484478}

\bibitem[BK05]{Bryant}
Roger~M. Bryant and Gregor Kemper, \emph{Global degree bounds and the transfer
  principle for invariants}, J. Algebra \textbf{284} (2005), no.~1, 80--90.
  \MR{MR2115005 (2005i:13007)}

\bibitem[dCP76]{ConciniProcesi}
C.~de~Concini and C.~Procesi, \emph{A characteristic free approach to invariant
  theory}, Advances in Math. \textbf{21} (1976), no.~3, 330--354. \MR{MR0422314
  (54 \#10305)}

\bibitem[DH00]{DomokosHegedus}
M{\'a}ty{\'a}s Domokos and P{\'a}l Heged{\H{u}}s, \emph{Noether's bound for
  polynomial invariants of finite groups}, Arch. Math. (Basel) \textbf{74}
  (2000), no.~3, 161--167. \MR{MR1739493 (2001a:13005)}

\bibitem[DK02]{Derksen2002Computational-i}
Harm Derksen and Gregor Kemper, \emph{Computational invariant theory},
  Invariant Theory and Algebraic Transformation Groups, I, Springer-Verlag,
  Berlin, 2002, Encyclopaedia of Mathematical Sciences, 130. \MR{MR1918599
  (2003g:13004)}

\bibitem[DK04]{DerksenKemperGlobal}
H.~Derksen and Gregor Kemper, \emph{On global degree bounds for invariants},
  Invariant theory in all characteristics, CRM Proc. Lecture Notes, vol.~35,
  Amer. Math. Soc., Providence, RI, 2004, pp.~37--41. \MR{MR2066457
  (2005a:13011)}

\bibitem[DKW08]{DraismaSeparating}
Jan Draisma, Gregor Kemper, and David Wehlau, \emph{Polarization of separating
  invariants}, Canad. J. Math. \textbf{60} (2008), no.~3, 556--571.
  \MR{MR2414957}

\bibitem[Fle00]{FleischmannNoether}
Peter Fleischmann, \emph{The {N}oether bound in invariant theory of finite
  groups}, Adv. Math. \textbf{156} (2000), no.~1, 23--32. \MR{MR1800251
  (2001k:13009)}

\bibitem[Fog01]{FogartyNoether}
John Fogarty, \emph{On {N}oether's bound for polynomial invariants of a finite
  group}, Electron. Res. Announc. Amer. Math. Soc. \textbf{7} (2001), 5--7
  (electronic). \MR{MR1826990 (2002a:13002)}

\bibitem[Kem03]{KemperCompRed}
Gregor Kemper, \emph{Computing invariants of reductive groups in positive
  characteristic}, Transform. Groups \textbf{8} (2003), no.~2, 159--176.
  \MR{MR1976458 (2004b:13006)}

\bibitem[Kem09]{KemperSeparating}
\bysame, \emph{Separating invariants}, J. Symbolic Comput. \textbf{44} (2009),
  no.~9, 1212--1222. \MR{MR2532166}

\bibitem[MFK94]{Mumford1994Geometric-invar}
D.~Mumford, J.~Fogarty, and F.~Kirwan, \emph{Geometric invariant theory}, third
  ed., Ergebnisse der Mathematik und ihrer Grenzgebiete (2) [Results in
  Mathematics and Related Areas (2)], vol.~34, Springer-Verlag, Berlin, 1994.
  \MR{MR1304906 (95m:14012)}

\bibitem[New78]{Newstead}
P.~E. Newstead, \emph{Introduction to moduli problems and orbit spaces}, Tata
  Institute of Fundamental Research Lectures on Mathematics and Physics,
  vol.~51, Tata Institute of Fundamental Research, Bombay, 1978. \MR{MR546290
  (81k:14002)}

\bibitem[Rob61]{Roberts}
Michael Roberts, \emph{On the {C}ovariants of a {B}inary {Q}uantic of the
  $n^{th}$ {D}egree}, The Quarterly Journal of Pure and Applied Mathematics
  \textbf{4} (1861), 168--178.

\bibitem[Sch91]{Schmid1991Finite-groups-a}
Barbara~J. Schmid, \emph{Finite groups and invariant theory}, Topics in
  invariant theory ({P}aris, 1989/1990), Lecture Notes in Math., vol. 1478,
  Springer, Berlin, 1991, pp.~35--66. \MR{MR1180987 (94c:13002)}

\bibitem[Sez02]{SezerNoether}
M{\"u}fit Sezer, \emph{Sharpening the generalized {N}oether bound in the
  invariant theory of finite groups}, J. Algebra \textbf{254} (2002), no.~2,
  252--263. \MR{MR1933869 (2003h:13007)}

\bibitem[Spr98]{SpringerNew}
T.~A. Springer, \emph{Linear algebraic groups}, second ed., Progress in
  Mathematics, vol.~9, Birkh\"auser Boston Inc., Boston, MA, 1998.
  \MR{MR1642713 (99h:20075)}

\bibitem[Weh06]{WehlauNoether}
David~L. Wehlau, \emph{The {N}oether number in invariant theory}, C. R. Math.
  Acad. Sci. Soc. R. Can. \textbf{28} (2006), no.~2, 39--62. \MR{MR2257602
  (2007h:13008)}

\end{thebibliography}
\bibliographystyle{amsalpha}

\end{document}